\documentclass[oneeqnum,onethmnum]{siamltex}

\usepackage{subeqn}
\usepackage{amsfonts}
\usepackage{subfigure}
\usepackage{graphicx}

\newcommand{\rr}{\mathbb{R}}
\newcommand{\nn}{\mathbb{N}}
\renewcommand{\d}{\mbox{d}}
\renewcommand{\.}{\cdot}
\newcommand{\e}{\epsilon}
\newcommand{\s}{\sigma}
\newcommand{\test}{\varphi}
\newcommand{\dualpair}[2]{\langle #1,#2 \rangle}
\newcommand{\norm}[2]{\|#1\|_{#2}}
\newcommand{\X}{\mathcal{X}}

\title{On the Effects of Bohm's Potential on a Stationary Macroscopic System of Self-Interacting Particles}

\author{Ren\'e Pinnau\thanks{Department of Technomathematics, University of Kaiserslautern, 67663 Kaiserslautern, Germany} \and Oliver Tse\thanks{Department of Technomathematics, University of Kaiserslautern, 67663 Kaiserslautern, Germany ({\tt tse@mathematik.uni-kl.de}).}}

\begin{document}

\maketitle

\begin{abstract}
 We consider a macroscopic model describing a system of self-gravitating particles. We study the existence and uniqueness of non-negative stationary solutions and allude the differences to results obtained from classical gravitational models. The problem is considered on a bounded domain up to three space dimension, subject to Neumann boundary condition for the particle density, and Dirichlet boundary condition for the self-interacting potential. Finally, we show numerical simulations that affirm our findings.
\end{abstract}



\pagestyle{myheadings}
\thispagestyle{plain}
\markboth{Ren\'e Pinnau and Oliver Tse}{Macroscopic Model for Quantum Gravitation}

\section{Introduction}
Consider a stationary macroscopic system of self-interacting particles with {\em Bohm's porential}, which describe the \emph{normalized density} $n\ge 0$,
\begin{subequations}\label{eq::system}
\begin{equation}
 -\e^2\sqrt{n}^{-1}\Delta\sqrt{n} + \log{n} - \s\Phi = F \;\;\mbox{in}\;\Omega,\quad
 \partial_\nu \sqrt{n} = 0 \;\;\mbox{on}\;\Gamma,
\end{equation}
the {\em quasi Fermi-level} $F$,
\begin{equation}
 -\mbox{div}\left(n\nabla F\right) = 0 \;\;\mbox{in}\;\Omega,\quad 
 n\,\partial_\nu F = 0 \;\;\mbox{on}\;\Gamma,
\end{equation}
and the \emph{potential} $\Phi$ due to self-interaction in a particle system,
\begin{equation}
 -\Delta \Phi = n \;\;\mbox{in}\;\Omega,\quad 
 \Phi = 0 \;\;\mbox{on}\;\Gamma,
\end{equation}
\end{subequations}
where $\nu$ is the outer normal to the convex, bounded domain $\Omega\subset\rr^d$, $d\le 3$ with Lipschitz boundary $\Gamma$, $\e>0$ is the scaled Planck constant and $|\s|\in[0,\infty)$ is the mass of the system of self-interacting particles, where sign($\s$) dictates the nature of the interaction involved. In this case, positive mass $\s>0$ would indicate the presence of self-attraction, while negative mass $\s<0$ indicates self-repulsion. By passing to the limit $\e\to 0$, we formally recover either the classical drift-diffusion equations ($\s<0$) \cite{abdallah} or a model for a system of self-gravitating particles ($\s>0$) \cite{biler}.

A simple observation of (\ref{eq::system}b) suggests that for any positive density $n>0$, we obtain a solution $F\in\rr$. In fact, any constant function is a solution. However, we shall see below that this constant solution is fixed due to the normality of $n$.

In order to bring this system of equations into a system similar to that of the classical equations for self-interacting particles (c.f.~\cite{abdallah,biler}), we introduce the \emph{quasi potential} $u:= \log n - F$. Assuming $n>0$, we insert $F=\log n - u$ into (\ref{eq::system}b) to obtain
\begin{equation}\label{eq::classical}
 -\Delta n + \mbox{div}\left(n\nabla u\right) = 0 \;\;\mbox{in}\;\Omega,\quad
 \partial_\nu n - n\,\partial_\nu u = 0 \;\;\mbox{on}\;\Gamma.
\end{equation}
At this point, one directly sees the resemblance to the classical equations for self-gravitating particles if we simply set $u=\Phi$ in (\ref{eq::classical}) and couple it with equation (\ref{eq::system}c) for the potential $\Phi$. For this reason, we call $u$ the quasi potential. Clearly, we may further rewrite (\ref{eq::classical}) in the equivalent form
\[
 -\mbox{div}\left(e^{u}\nabla\left(ne^{-u}\right) \right) = 0 \;\;\mbox{in}\;\Omega,\quad 
 \partial_\nu\left(ne^{-u}\right) = 0 \;\;\mbox{on}\;\Gamma.
\]
This allows us to relate $n$ and $u$ via $n=\alpha e^{u}$ for some constant $\alpha>0$. Since $n$ is normalized, i.e., $\int_\Omega n\,\d{x}=1$, we deduce that $\alpha = \|e^u\|_{L^1(\Omega)}^{-1}$. Notice that, by fixing $\alpha$, we also fix the quasi Fermi-level, which is explicitly given by $F=\log(\alpha)$.

Introducing this into (\ref{eq::system}) leads to the coupled system for $(u,\Phi)$ given by an elliptic equation with natural gradient growth for $u$
\begin{subequations}\label{eq::system2}
 \begin{equation}
  -\frac{\e^2}{2}\Delta u + u = \frac{\e^2}{4}|\nabla u|^2 + \s\Phi \;\;\mbox{in}\;\Omega,\quad 
  \partial_\nu u = 0 \;\;\mbox{on}\;\Gamma,
 \end{equation}
and the equation for the potential $\Phi$,
 \begin{equation}
  -\Delta \Phi = \|e^u\|_{L^1(\Omega)}^{-1} e^u \;\;\mbox{in}\;\Omega,\quad
  \Phi = 0 \;\;\mbox{on}\;\Gamma.
 \end{equation}
\end{subequations}
Note that system (\ref{eq::system2}) is equivalent to system (\ref{eq::system}) if $n>0$, or equivalently, if $u$ is an essentially bounded function. The existence of bounded weak solutions $(u,\Phi)$ will be shown for (\ref{eq::system2}), thereby implying the existence of solutions for (\ref{eq::system}). 

We introduce the short hand $\X$ to denote the space 
\[
 \X:=H^1(\Omega)\cap L^\infty(\Omega).
\]

\begin{theorem}\label{thm::main}
 Let $d\in\{2,3\}$ and the mass $|\s|\in[0,+\infty)$ be given. Then problem (\ref{eq::system2}) has a solution $(u,\Phi)\in \X\times \X$. Consequently, $(n,F,\Phi)$ with 
 \[
  \sqrt{n} = \|e^u\|_{L^1(\Omega)}^{-\frac{1}{2}}e^{u/2}\in \X\quad\mbox{and}\quad F = \log n - u = -\log \|e^u\|_{L^1(\Omega)}\in\rr,
 \]
 is a solution of (\ref{eq::system}). Moreover, there exists $\theta\in(0,1)$, such that $\theta\leq n\leq 1/\theta$.
\end{theorem}

Observe that $\s>0$ can be chosen arbitrarily large as opposed to classical self-gravitating particles, where a threshold for existence exists. In this sense, system (\ref{eq::system}) can be thought of as a regularization of the classical self-gravitating system.

Let us discuss the techniques used to show existence of solutions for system (\ref{eq::system2}). The solvability of (\ref{eq::system2}a) and its variants with homogeneous Dirichlet boundary data were shown in the papers \cite{dall,ferone,montenegro}. Moreover, if $\Phi\in L^p(\Omega)$ with $p>d/2$, then $u\in \X$. Adopting the methods used in \cite{dall}, we show in Section~\ref{sec::2} that this holds true also for homogeneous Neumann boundary data. In this case, systems (\ref{eq::system}) and (\ref{eq::system2}) are equivalent. Furthermore, we obtain from \cite{cassani} the following sharp estimates for functions in the space $W^{2,1}_{\Delta,0}=\overline{\mathcal{D}(\Omega)}^{\|\Delta \.\|_{L^1(\Omega)}}$.

\begin{proposition}\label{prop::estimate}
 For any $f\in W^{2,1}_{\Delta,0}$ with $f\ge 0$ in $\Omega$ we have the estimates
 \begin{eqnarray*}
  \|f\|_{L_{\exp}(\Omega)} &\leq& (8\pi)^{-1}\|\Delta f\|_{L^1(\Omega)},\quad\hspace*{0.4em} d=2,\\
  \|f\|_{L^{\frac{d}{d-2},\infty}(\Omega)} &\leq& (2\gamma_d)^{-1}\|\Delta f\|_{L^1(\Omega)},\quad d\geq 3,
 \end{eqnarray*}
with $\gamma_d = \omega_{d-1}^{2/d}(d-2)d^{\frac{d-2}{d}}$, where $\omega_{d-1}$ is the measure of the unit sphere in $\rr^d$. The constants given above are the best possible, independently of the domain.
\end{proposition}

Here we used $L_{\exp}(\Omega)$ to denote the Zygmund space, whose elements $f$ satisfy $\int_\Omega e^{\lambda f}\,\d{x}<\infty$ for some $\lambda=\lambda(f)>0$, and $L^{p,\infty}(\Omega)$ to denote the classical weak-$L^p$ space (see~\cite{bennett}). These spaces are related to each other and to the $L^p$ spaces by the following continuous embeddings for $1<p<\infty$,
\begin{equation}\label{eq::embeddings}
 L^\infty(\Omega)\hookrightarrow L_{\exp}(\Omega)\hookrightarrow L^p(\Omega)\hookrightarrow L^{p,\infty}(\Omega)\hookrightarrow L^1(\Omega).
\end{equation}
The estimates in Proposition~\ref{prop::estimate} necessarily implies that the $L^q$-norm of $\Phi$, for any $q\in(1,\frac{d}{d-2})$, is uniformly bounded for $d\ge 2$. In particular, we have
\begin{eqnarray}\label{eq::estimate}
  \|\Phi\|_{L^q(\Omega)} \leq c_q\mu_d^{-1},\quad\mbox{with}\quad \mu_d = \left\lbrace
 \begin{array}{ll}
  8\pi, & d=2, \\
  2\gamma_d, & d\ge 3,
 \end{array}\right.
\end{eqnarray}
where $c_q$ is the embedding constant obtained from (\ref{eq::embeddings}). 

Therefore, an application of the Schauder fixed point theorem on a self-mapping for the potential $\Phi$ leads to its existence in $L^q(\Omega)$ for some $q\in(\frac{d}{2},\frac{d}{d-2})$, and consequently also for $u$ and $n$, which yields the existence result.
For sufficiently $\s$ we further obtain uniqueness of solutions, given by the following result.

\begin{theorem}\label{thm::unique}
 Let $\theta\in(0,1)$ and $n\in\X$ with $\theta \leq n\leq 1/\theta$. There exist constants $c_0=c_0(d,\Omega)>0$ and $c_1=c_1(d,\Omega,\theta)>0$ such that for 
 \[
  |\s|<\mu_d\sqrt{c_0^2 + 2\e^2c_1^2},
 \] 
 the solutions of (\ref{eq::system}) are equal almost everywhere in $\Omega$.
\end{theorem}

To our knowledge, this estimate for uniqueness appears to be new for both attractive and repulsive potentials. For $\s<0$, it is known that in the classical case $\e=0$, uniqueness depends on the smallness of the applied voltage. As a matter of fact, the performance of many semi-conductor devices (thyristors) depends on the existence of multiple solutions (c.f.~\cite{pinnau} and references therein).

Unfortunately, this estimate is not sharp, since for $d=2$ and $\s>0$, it is known that uniqueness is valid for $\s<\mu_2$ (c.f.~\cite{caglioti2}). Nevertheless, the above estimate provides a convenient relationship between uniqueness and the Bohm potential ($\e>0$).
%In particular, uniqueness holds for all $\s<0$, which is a well-known result for self-consistent quantum drift-diffusion equations ()

\section{Elliptic equation with natural gradient growth}\label{sec::2}

We begin by providing results for the subproblem (\ref{eq::system2}a) required to prove Theorem~\ref{thm::main}.

\begin{definition}
 A function $u\in\X$ is said to be a solution of (\ref{eq::system2}a) if it satisfies
 \begin{equation}\label{eq::weak}
  \frac{\e^2}{2}\int_\Omega\nabla u\.\nabla\test\,\emph{\d}{x} + \int_\Omega u\,\test\,\emph{\d}{x} = \frac{\e^2}{4}\int_\Omega|\nabla u|^2\test\,\emph{\d}{x} + \s\int_\Omega \Phi\test\,\emph{\d}{x}
 \end{equation}
 for every function $\test\in\X$.
\end{definition}

\begin{theorem}\label{thm::existence}
 Suppose $\Phi\in L^p(\Omega)$, $p>d/2$. Then there is a solution $u\in\X$ of problem (\ref{eq::weak}). Furthermore, we have $e^{u/2}\in\X$.
\end{theorem}

Before proving the theorem, we state two results regarding the regularity of the solution $u$, whose proofs can be found in the appendix.

\begin{lemma}\label{lem::linfty&h1}
 Let $u$ be solution of (\ref{eq::system2}a) with $\Phi\in L^p(\Omega)$, $p>d/2$. Then 
 \begin{enumerate}
  \item for any $\lambda\geq 0$ there exist constants $K_1,K_2=K_2(\lambda)>0$ such that
   \begin{equation}
    \|u\|_{L^\infty(\Omega)} \leq K_1\quad\mbox{and}\quad \|e^{\lambda u}\|_{L^\infty(\Omega)}\leq K_2.
   \end{equation}
   In particular, $e^{\lambda u}\in L^\infty(\Omega)$ for every $\lambda\geq 0$.
   
   \item there exist constants $M_1,M_2>0$ such that
    \begin{equation}
     \|u\|_{H^1(\Omega)}\leq M_1\quad\mbox{and}\quad \|e^{|u|/2}\|_{H^1(\Omega)}\leq M_2.
    \end{equation}
 \end{enumerate}
 
\end{lemma}

{\em Proof of Theorem~\ref{thm::existence}}. Define $\phi\colon\rr\to\rr$ by $\phi(s) = (e^{|s|}-1)\,\mbox{sign}(s)$ and the cut-off function $T_k\colon \rr\to\rr$ by $T_k(s) = \max\{-k,\min\{s,k\}\}$ for some $k\in\rr$.
 We begin by considering the auxiliary problem for $u_k\in\X$:
 \begin{equation}\label{eq::weak_aux}
  \frac{\e^2}{2}\int_\Omega\nabla u_k\.\nabla\test\,\d{x} + \int_\Omega u_k\,\test\,\d{x} = \frac{\e^2}{4}\int_\Omega T_k(|\nabla u_k|^2)\,\test\,\d{x} + \s\int_\Omega T_k(\Phi)\test\,\d{x}.
 \end{equation}
 Since the right-hand side is bounded, the existence of a bounded solution for (\ref{eq::weak_aux}), $k\in\nn$, may be deduced from classical results (see for example \cite{leray} for the existence and \cite{kinder} for the boundedness). Due to the fact that $T_k(|\nabla u_k|^2)\leq |\nabla u_k|^2$ and $|T_k(\Phi)|\leq |\Phi|$ along with Lemma~\ref{lem::linfty&h1}, there exists a function $u\in\X$ such that
 \[
  u_k \rightharpoonup u\;\;\mbox{in}\;\;H^1(\Omega)\quad\mbox{and}\quad u_k\rightharpoonup^* u\;\;\mbox{in}\;\;L^\infty(\Omega).
 \]
 In order to pass to the limit in (\ref{eq::weak_aux}), we still need to show that $\nabla u_k\rightarrow \nabla u$ in $L^2(\Omega)$, i.e., the strong convergence of the gradients of $u_k$ in $L^2(\Omega)$. To do so, we test (\ref{eq::weak_aux}) with $\test_k = \phi(u_k-u)$ to obtain
 \begin{eqnarray}\label{eq::ineq3}
  \frac{\e^2}{2}\int_\Omega\nabla u_k\.\nabla(u_k-u)\test_k'\,\d{x} + \int_\Omega u_k\,\test_k\,\d{x} &&\\
  &&\hspace*{-8em}\leq \frac{\e^2}{4}\int_\Omega|\nabla u_k|^2|\test_k|\,\d{x} + \s\int_\Omega |\Phi||\test_k|\,\d{x}.\nonumber
 \end{eqnarray}
 For the first term on the left-hand side we have
 \[
  \int_\Omega\nabla u_k\.\nabla(u_k-u)\test_k'\,\d{x} = \int_\Omega|\nabla(u_k-u)|^2\test_k'\,\d{x} + \int_\Omega\nabla u\.\nabla(u_k-u)\test_k'\,\d{x}.
 \]
 As for the second term on the left hand-side, we have
 \[
  \int_\Omega u_k\,\test_k\,\d{x} = \int_\Omega |u_k-u|(e^{|u_k-u|}-1) + u\,\test_k\,\d{x} \geq \int_\Omega |u_k-u|^2\,\d{x} + \int_\Omega u\,\test_k\,\d{x}.
 \]
 For the first term on the right-hand side we have
 \begin{eqnarray*}
  \int_\Omega|\nabla u_k|^2|\test_k|\,\d{x} &=& \int_\Omega \nabla u_k\.\nabla(u_k-u)\,|\test_k|\,\d{x} + \int_\Omega\nabla u_k\.\nabla u\,|\test_k|\,\d{x}\\
  &&\hspace*{-6em}= \int_\Omega |\nabla(u_k-u)|^2|\test_k|\,\d{x} + \int_\Omega \nabla u\.\nabla(u_k-u)\,|\test_k|\,\d{x} + \int_\Omega\nabla u_k\.\nabla u\,|\test_k|\,\d{x}
 \end{eqnarray*}
 Due to the compact embedding $H^1(\Omega)\hookrightarrow L^2(\Omega)$, we get a convergent subsequence, denoted again by $\{u_k\}$, such that $u_k\to u$ in $L^2(\Omega)$. Consequently we obtain yet another subsequence, denoted again by $\{u_k\}$, such that $u_k(x)\to u(x)$ for a.e.~$x\in\Omega$, which implies the almost everywhere convergences
 \[
  |\test_k(x)|\to 0 \quad\mbox{and}\quad \test_k'(x)\to 1\quad\mbox{for a.e. }x\in\Omega.
 \]
 From Lebesgue's dominated convergence for these sequences and their boundedness in $L^\infty(\Omega)$, we have the strong convergences
 \[
  \nabla u\,|\test_k|\to 0\quad\mbox{and}\quad \nabla u\,\test_k'\to \nabla u\;\;\mbox{in}\;\;L^2(\Omega),\quad u\,|\test_k|,\;|\Phi||\test_k|\to 0\;\;\mbox{in}\;\;L^1(\Omega).
 \]
 Passing to the limit in (\ref{eq::ineq3}) yields
 \[
  \int_\Omega|\nabla(u_k-u)|^2\test_k'\,\d{x} + \int_\Omega|\nabla(u_k-u)|^2\,\d{x} \to 0,
 \]
 which necessarily implies that $\nabla u_n\to \nabla u$ in $L^2(\Omega)$. Therefore, passing to the limit in (\ref{eq::weak_aux}) yields the solution $u\in\X$ satisfying (\ref{eq::weak}). The fact that $e^{\lambda u}\in\X$ for every $\lambda\geq 0$ follows directly from Lemma~\ref{lem::linfty&h1}.\qquad\endproof

\section{Proof of Theorem~\ref{thm::main}}
 As drafted out above, we use the Schauder fixed point theorem (c.f.~\cite[Corollary~11.2]{gilbarg}) to facilitate the proof. We define the closed, convex and bounded subset of $L^q(\Omega)$
 \[
  M=\left\lbrace w\in L^q(\Omega)\,\left|\,\|w\|_{L^q(\Omega)} \leq c_q\mu_d^{-1},\; w\ge 0\right.\right\rbrace,\quad 
 \]
 for some $q\in(\frac{d}{2},\frac{d}{d-2})$ and with $\mu_d$ as given in (\ref{eq::estimate}).

 For a given $w\in M$, we consider the auxiliary problem for $\Phi$ given by
 \begin{subequations}\label{eq::auxiliary}
  \begin{equation}
   -\frac{\e^2}{2}\Delta u + u = \frac{\e^2}{4}|\nabla u|^2 + \s w \;\;\mbox{in}\;\Omega,\quad
   \partial_\nu u = 0 \;\;\mbox{on}\;\Gamma,
  \end{equation}
  \begin{equation}
   -\Delta \Phi = \|e^u\|_{L^1(\Omega)}^{-1} e^u \;\;\mbox{in}\;\Omega,\quad 
   \Phi = 0 \;\;\mbox{on}\;\Gamma.
  \end{equation}
 \end{subequations}
 This induces a compact mapping 
 \[
  \mbox{H}\colon L^q(\Omega)\to L^q(\Omega);\; w\mapsto \Phi,
 \]
 simply due to the continuity of the solution operators between their respective spaces and the compact embedding $H^1(\Omega)\hookrightarrow L^q(\Omega)$. Indeed, since $w\in L^q(\Omega)$, we obtain a solution $u\in\X$ of (\ref{eq::auxiliary}a) as a result of Theorem~\ref{thm::existence}. Inserting $u$ into (\ref{eq::auxiliary}b) and solving for $\Phi$ gives us $\Phi\in\X$ with $\Phi\ge 0$ due to standard theory of elliptic equations.

 Furthermore, we have $\mbox{H}\colon M\to M$ simply due to the estimates in (\ref{eq::estimate}). A direct application of the Schauder fixed point theorem concludes the proof.\qquad\endproof

\section{Proof of Theorem~\ref{thm::unique}}
 Let $A\colon\X\to H^1(\Omega)^*$ denote the operator defined by
 \[
  \dualpair{A(n)}{\test} = \int_\Omega\left(-\e^2\sqrt{n}^{-1}\Delta\sqrt{n} + \log{n}\right)\test\,\d{x}\qquad\forall\,\test\in H^1(\Omega).
 \]
 Then the weak formulation of (\ref{eq::system}a) can be written as
 \[
  n\in\X:\quad\dualpair{A(n) - \s\Phi}{\test} = \dualpair{F}{\test}\qquad\forall\,\test\in H^1(\Omega).
 \]

 With similar arguments to those by Pinnau, Unterreiter in \cite[Theorem~26]{pinnau} the operator $A$ is well-defined. Moreover, for a fixed $\test\in H^1(\Omega)$, the Gate\^aux derivative of $\dualpair{A(\.)}{\test}\colon \X\to\rr$ at a point $n\in\X$ in any direction $h\in \X$ exists and is given by
 \[
  \dualpair{A'(n)[h]}{\test} = -\frac{\e^2}{2}\int_\Omega\left(\frac{\Delta h}{n} - \frac{\Delta n}{n^2}h - \frac{\nabla n\.\nabla h}{n^2} + \frac{|\nabla n|^2}{n^3}h\right)\test\,\d{x} + \int_\Omega\frac{h}{n}\,\test\,\d{x}.
 \]

 Now let $(n_i,F_i,\Phi_i)\in [\X]^3$, $i=1,2$, be two solutions of (\ref{eq::system}) and set the difference to be $(\delta n,\delta F,\delta \Phi):=(n_1-n_2, F_1-F_2, \Phi_1-\Phi_2) \in [\X]^2$. In particular we have
 \begin{equation}\label{eq::weak_i}
  n_i\in\X:\quad\dualpair{A(n_i)}{\test} = \dualpair{\s\Phi_i + F_i}{\test}\quad\forall\,\test\in H^1(\Omega)
 \end{equation}
 Setting $\test=\delta n$ we obtain by subtraction from (\ref{eq::weak_i})
 \[
  \dualpair{A(n_1)-A(n_2)}{\delta n} = \s\dualpair{\delta \Phi}{\delta n} + \dualpair{\delta F}{\delta n} = \s\dualpair{\delta \Phi}{\delta n},
 \]
 where we used the fact that $\delta F\in\rr$ and $\int_\Omega \delta n\,\d{x} = 0$.

 Let $n_\tau = n_1 - \tau\delta n$, $\tau\in[0,1]$ be the convex combination of $n_1$ and $n_2$. Since the function $[0,1]\ni\tau\mapsto \dualpair{A(n_\tau)}{\delta n}\in\rr$ is differentiable, we obtain, by the mean value theorem, a $\tau\in(0,1)$ such that
 \[
  \e^2\int_\Omega n_\tau\left|\nabla\left(\frac{\delta n}{n_\tau}\right)\right|^2\,\d{x} + \int_\Omega \frac{|\delta n|^2}{n_\tau}\,\d{x} = \dualpair{A'(n_\tau)[\delta n]}{\delta n} = \s\int_\Omega \delta \Phi\delta n\,\d{x}.
 \]

 Now, in order to estimate the first term on the left-hand side from below, we use a variant of the result obtain in \cite[Lemma~24]{pinnau}.
 
 \begin{proposition}\label{prop::bohm}
  Let assumption (A) hold. Then there exists for any $\beta\in\rr$ and $\theta\in(0,1)$ a constant $c_K=c_K(\Omega,\theta,s)>0$ such that for any $n\in\X$ with $\theta\leq n\leq 1/\theta$ and any $\test\in \X$ with $\int_\Omega \test\,\d{x}=0$:
  \begin{equation}
   \int_\Omega n\left|\nabla\left(\frac{\test}{n}\right)\right|^2\,\emph{\d}{x}\geq c_K^2\|\test\|_{L^s(\Omega)}^2,
  \end{equation}
  where $s\in[1,\infty)$ such that the Sobolev embedding $H^1(\Omega)\hookrightarrow L^s(\Omega)$ holds.
 \end{proposition}

 For the second term on the left-hand side, we apply H\"older's inequality to obtain
 \[
  \int_\Omega |\delta n|\,\d{x} = \int_\Omega \left(\frac{|\delta n|}{\sqrt{n_\tau}}\right)\sqrt{n_\tau}\,\d{x} \leq \left(\int_\Omega \frac{|\delta n|^2}{n_\tau}\,\d{x}\right)^\frac{1}{2},
 \]
 where we used the fact that $\int_\Omega n_\tau\,\d{x}=1$. Altogether we have the estimate
 \begin{equation}\label{eq::coercive}
  \dualpair{A'(n_\tau)[\delta n]}{\delta n} \geq \e^2c_K^2\|\delta n\|_{L^s(\Omega)}^2 + \|\delta n\|_{L^1(\Omega)}^2
 \end{equation}
 for some $s\in[1,\infty)$ satisfying the requirements of Proposition~\ref{prop::bohm}.

 Note that by subtraction, $\delta \Phi\in\X$ solves the problem
 \[
  -\Delta \delta \Phi = \delta n \;\;\mbox{in}\;\Omega,\quad \delta \Phi = 0 \;\;\mbox{on}\;\Gamma.
 \]
 As before, the estimates in Proposition~\ref{prop::estimate} yield for $q\in(\frac{4d}{d+2},\frac{d}{d-2})$ the inequality
 \[
  \|\delta \Phi\|_{L^q(\Omega)} \leq c_q\mu_d^{-1}\|\delta n\|_{L^1(\Omega)}.
 \]
 Due to the constraint placed on $q$, we have the Sobolev embedding $H^1(\Omega)\hookrightarrow L^{q'}(\Omega)$, where $q'=\frac{q}{q-1}$, thereby allowing us to choose $s=q'$ in (\ref{eq::coercive}). 

 Putting together all the inequalities obtained above, we have
 \[
  \e^2c_K^2\|\delta n\|_{L^{q'}(\Omega)}^2 + \|\delta n\|_{L^1(\Omega)}^2 \leq c_q\mu_d^{-1}|\s|\|\delta n\|_{L^1(\Omega)}\|\delta n\|_{L^{q'}(\Omega)}.
 \]
  Applying Young's inequality we obtain
 \[
  \e^2c_K^2\|\delta n\|_{L^{q'}(\Omega)}^2 + \|\delta n\|_{L^1(\Omega)}^2 \leq \frac{1}{2}\|\delta n\|_{L^1(\Omega)}^2 + \frac{(c_q\mu_d^{-1}|\s|)^2}{2}\|\delta n\|_{L^{q'}(\Omega)}^2,
 \]
 Using the continuous embedding $L^{q'}(\Omega)\hookrightarrow L^1(\Omega)$, we finally arrive at
 \[
  \left(\frac{2\e^2c_K^2 - c_q^2\mu_d^{-2}|\s|^2}{c_{q'}^2} + 1\right)\norm{\delta n}{L^1(\Omega)}^2\leq 0.
 \]
 In conclusion, for sufficiently small mass
 \[
  |\s| < \mu_d\sqrt{(c_{q'}/c_q)^2 + 2\e^2(c_K/c_q)^2} = \mu_d\sqrt{c_0^2 + 2\e^2c_1^2},
 \]
 we obtain uniqueness for $n$, and consequently for $\Phi$.\qquad\endproof

% Unfortunately, this estimate for uniqueness is not sharp. It is known for the case $d=2$ (see \cite{caglioti2}) that uniqueness holds for any $\s<\mu_2$.

\section{The semi-classical limit} This result is well-known for the case $\s<0$, i.e., for the quantum drift-diffusion equations (c.f.~\cite{abdallah}). Therefore, we restrict ourselves to the case $\s>0$. Furthermore, we consider only the case $d=2$. 

According to \cite{caglioti1} the \emph{free energy} functional
 \[
  \mathcal{E}_0(n) = \int_\Omega n(\log{n}-1)\,\d{x} - \frac{\s}{2}\int_\Omega n\Phi\,\d{x},\quad  \s<\mu_2=8\pi,
 \]
 where $\Phi$ is a solution of the Poisson problem (\ref{eq::system}c) attains a minimum $n_0$ in the set $$\mathcal{P}=\left\lbrace n\in L^1(\Omega)\;|\; n\ge 0,\; \int_\Omega n\,\d{x}=1,\; n\log n\in L^1(\Omega)\right\rbrace.$$
 Moreover, the minimizer is unique (c.f.~\cite[Theorem~3.2]{caglioti2}). We denote its associated potential by $\Phi_0$ and recall the relationship
 \begin{equation}\label{eq:n_0}
  -\Delta \Phi_0 = n_0 = \|e^{\s\Phi_0}\|_{L^1(\Omega)}e^{\s\Phi_0},
 \end{equation}
 which solves the classical stationary system of self-gravitating particles (c.f.~\cite{biler,suzuki}),
 \begin{subequations}\label{eq:classical}
 \begin{eqnarray}
  -\Delta n + \s\mbox{div}\left(n\nabla \Phi\right) = 0\hspace*{0.1em} \;\;\mbox{in}\;\Omega,&\qquad 
  \partial_\nu n - \s n\,\partial_\nu\Phi = 0 \;\;\mbox{on}\;\Gamma, \\
  -\Delta \Phi = n \;\;\mbox{in}\;\Omega,&\qquad\hspace*{5.2em} 
  \Phi = 0 \;\;\mbox{on}\;\Gamma,
 \end{eqnarray}
 \end{subequations}

 Now consider the energy functional associated to (\ref{eq::system})
 \[
  \mathcal{E}_\e(n) = \e^2 \mathcal{F}(n) + \mathcal{E}_0(n),\quad \s<\mu_2,
 \]
 where $\mathcal{F}$ is the \emph{Fisher information} given by
 \[
  \mathcal{F}(n) = \int_\Omega |\nabla\sqrt{n}|^2\,\d{x}.
 \]
 It is easy to see that $\mathcal{E}_\e$ is weakly lower semicontinuous, strictly convex and coercive on $\mathcal{P}_\e=\left\lbrace n\in \mathcal{P}\;|\;\sqrt{n}\in H^1(\Omega)\right\rbrace$. Indeed, $\mathcal{E}_0$ is equivalent to the functional
 \[
  \mathcal{G}(\Phi) = \frac{\s}{2}\int_\Omega|\nabla\Phi|^2 - \log\left(\int_\Omega e^{\s\Phi}\,\d{x}\right) - 1,\quad \s<\mu_2,
 \]
 which is uniformly bounded from below for all $\Phi\in H_0^1(\Omega)$ due to Moser \cite{moser}. Therefore, it attains a unique minimum $n_\e$ in the set $\mathcal{P}_\e$, and in particular in $\mathcal{P}$.

\begin{theorem}\label{thm:limit}
 Let $n_0$ and $n_\e$ be solutions of the problems 
 \[
  \min_{n\in\mathcal{P}} \mathcal{E}_0(n)\quad\mbox{and}\quad \min_{n\in\mathcal{P}_\e} \mathcal{E}_\e(n) 
 \]
 respectively. Then there exists $(n_*,\Phi_*,F_*)\in \X\times\X\times \rr$, which solves the classical self-gravitation system (\ref{eq:classical}) such that the following convergences hold for $\e\to 0_+$:
 \begin{eqnarray*}
  &\sqrt{n_\e} \rightharpoonup \sqrt{n_*}\;\;\mbox{in}\;\;H^1(\Omega),\quad \Phi_\e\to\Phi_*\;\;\mbox{in}\;\;H^1(\Omega),\quad
  u_\e \to \s\Phi_*\;\;\mbox{in}\;\;L^2(\Omega),& \\ &F_\e\to F_*=-\log \|e^{\s\Phi_*}\|_{L^1(\Omega)}\;\;\mbox{in}\;\;\rr.&
 \end{eqnarray*}
 Furthermore, if $n_0$ is a unique minimizer of $\mathcal{E}_0$, then $n_*\equiv n_0$ and $\Phi_*\equiv \Phi_0$.
\end{theorem}
\begin{proof}
 We begin by showing that $\{\sqrt{n_\e}\}\subset H^1(\Omega)$ is bounded. Indeed, since $n_\e$ is a minimum of $\mathcal{E}_\e$, we have
 \[
  \e^2\mathcal{F}(n_\e) + \mathcal{E}_0(n_\e) = \mathcal{E}_\e(n_\e) \le \mathcal{E}_\e(n_0) = \e^2\mathcal{F}(n_0) + \mathcal{E}_0(n_0),
 \]
 but $\mathcal{E}_0(n_0)\le \mathcal{E}_0(n_\e)$ since $n_0$ is a minimum of $\mathcal{E}_0$. Hence, $\mathcal{F}(n_\e)\le \mathcal{F}(n_0)$ for all $\e>0$, which was to be shown. We can then extract a subsequence, denoted again by $n_\e$, such that $\sqrt{n_\e}\rightharpoonup \sqrt{n_*}$ in $H^1(\Omega)$ and $n_\e\to n_*$ in $L^2(\Omega)$ %and $n_\e(x)\to n_*(x)$ for a.e.~$x\in\Omega$ 
 for some $n_*\in \mathcal{P}$, where the second convergence follows from the compact Sobolev embedding $H^1(\Omega)\hookrightarrow L^4(\Omega)$. Furthermore, we have 
 \[
  \mathcal{E}_0(n_0) \le \liminf_{\e\to 0_+} \mathcal{E}_0(n_\e) \le \liminf_{\e\to 0_+} \mathcal{E}_\e(n_\e) \le \limsup_{\e\to 0_+} \mathcal{E}_\e(n_\e) \le \mathcal{E}_0(n_0),
 \]
 which implies that $\mathcal{E}_0(n_0)=\lim_{\e\to 0_+}\mathcal{E}_\e(n_\e)$. On the other hand, by the weakly lower $L^2(\Omega)$-semicontinuity of the functional $\mathcal{E}_0$,
 \[
  \mathcal{E}_0(n_*) \le \liminf_{\e\to 0_+}\mathcal{E}_0(n_\e) \leq \limsup_{\e\to 0_+}\mathcal{E}_\e(n_\e)=\mathcal{E}_0(n_0).
 \]
 Therefore, $n_*$ is a minimizer of $\mathcal{E}_0$. The strong convergence $\Phi_\e\to \Phi_*:=(-\Delta_0)^{-1}n_*$ in $H^1(\Omega)$ follows easily from the strong convergence $n_\e\to n_*$ in $L^2(\Omega)$, due to the continuity of the solution operator $(-\Delta_0)^{-1}$.
 
%  Recall that, due to (\ref{eq:n_0}), we have the representation $\log n_0 - \s\Phi_0 = F_0$. 

 Now consider the Euler-Lagrange equation associated to $\mathcal{E}_\e$, i.e., the variational formulation of (\ref{eq::system}a), given by
 \[
   \e^2\int_\Omega\nabla\sqrt{n_\e}\.\nabla\test\,\d{x} + \int_\Omega (\log n_\e - \s\Phi_\e)\test\sqrt{n_\e}\,\d{x} = \int_\Omega F_\e\test \sqrt{n_\e}\,\d{x}\quad\forall \test\in H^1(\Omega).
 \]
 Similarly we have the Euler-Lagrange equation associated to $\mathcal{E}_0$, given by
 \[
  \int_\Omega (\log n_* - \s\Phi_*)\test\,\d{x} = \int_\Omega F_*\test\,\d{x}\quad\forall \test\in L^1(\Omega),
 \]
 where $F_*\in\rr$ is the Lagrange multiplier for the constraint $\int_\Omega n_*\,\d{x}=1$ with
 \begin{equation}\label{thm:limit:eq:1}
  F_*=-\log \|e^{\s\Phi_*}\|_{L^1(\Omega)},\;\;\mbox{since}\;\; \log n_* = \s\Phi_* + F_* \in L^\infty(\Omega).
 \end{equation}
 Therefore, by testing the former variational formulation with $\test=\sqrt{n_\e}$ and the latter with $\test=n_\e$, and taking the difference of the resulting equations, we obtain
 \[
   \e^2\mathcal{F}_\e(n_\e) + \int_\Omega \log (n_\e/n_*)n_\e\,\d{x} + \s\int_\Omega (\Phi_* - \Phi_\e)n_\e\,\d{x} = (F_\e-F_*),
 \]
 where we have used the fact that $F_\e,F_*\in\rr$ and $\int_\Omega n_\e\,\d{x}=1$. Due to the convergences derived above, we conclude that $F_\e \to F_*$ in $\rr$. At this point, it is easy to see from (\ref{thm:limit:eq:1}) that the pair $(u_*,\Phi_*)$ solves (\ref{eq:classical}a). To see that it also solves (\ref{eq:classical}b), we notice that $\Phi_*$ is a minimizer of $\mathcal{G}$ and that (\ref{eq:classical}b) is simply the Euler-Lagrange equation associated to $\mathcal{G}$. Hence, $(n_*,\Phi_*)\in \X\times \X$ is indeed a solution of (\ref{eq:classical}).
 
 Furthermore, from the representation $u_\e = \log n_\e - F_\e$ we obtain the strong convergence $u_\e\to \s\Phi_*$ in $L^2(\Omega)$. Indeed, by taking the difference of the two representations, multiplying the resulting equation with $(u_\e - \s\Phi_*)$ and integrating over $\Omega$, we obtain
 \begin{eqnarray*}
  \int_\Omega |(u_\e - \s\Phi_*)|^2\,\d{x} &=& \int_\Omega \log(n_\e/n_*)(u_\e-\s\Phi_*)\,\d{x} + \int_\Omega (F_\e-F_*)(u_\e-\s\Phi_*)\,\d{x} \\
  &\le& \left(\|\log(n_\e/n_*)\|_{L^2(\Omega)} + |\Omega|^{\frac{1}{2}}|F_\e-F_*|\right)\|u_\e-\s\Phi_*\|_{L^2(\Omega)},
 \end{eqnarray*}
 which clearly yields the required convergence.

 Finally, if the minimizer $n_0$ is unique, then $n_*\equiv n_0$ and consequently $\Phi_*\equiv \Phi_0$.
\end{proof}

\section{Numerical Simulations}
In this section we show two numerical simulations that validate the theoretical results obtained above. In both cases, we considered the unit disk $D\subset\rr^2$ with scaled Planck constant $\e = 1\times 10^{-3}$. An adaptive finite element method was used to solve the coupled problem (\ref{eq::system2}) for $(u,V)$ iteratively. The outer iteration consist of a Picard iteration procedure for the potential $V$ and the inner iteration consists of Newton's method to solve (\ref{eq::system2}a) for the quasi potential $u$.

{\bf Case 1:} The first case pertains to the existence of stationary states for system~(\ref{eq::system2}) with a large mass ($\s=10\pi$), which clearly exceeds the threshold ($\s=8\pi$) of existence in the classical setting. The numerical results are shown in Figure~\ref{fig::exist}.
One clearly observes the similarity of $u$ and $V$. This similarity shows that the quasi potential $u$ is a slight perturbation of the potential $V$ when the scaled Planck constant is small.
\begin{figure}[ht]\centering
 \subfigure[Quasi potential u]{
  \includegraphics[width=17em]{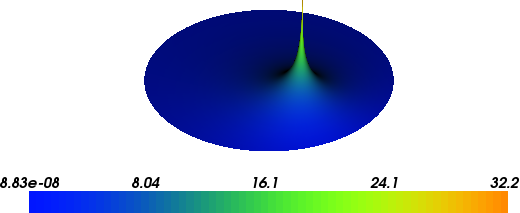}
  \label{fig::exist::quasi}
 }
 \subfigure[Potential V]{
  \includegraphics[width=17em]{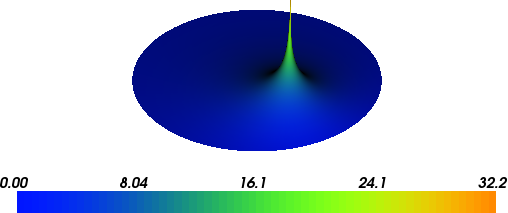}
  \label{fig::exist::potential}
 }
 \label{fig::exist}
 \caption{A stationary solution $(u,V)$ for $\s=10\pi$ with $F=-20.188$}
\end{figure}

{\bf Case 2:} The second case corresponds to the non-uniqueness of stationary states when their quasi Fermi-levels $F$ are identical and their mass $\s$ exceeds the threshold given in Theorem~\ref{thm::unique}. As in Case~1, we set $\s=10\pi$. By shifting the position of the starting value for the iteration procedure in Case~1, we obtained another solution with the same quasi Fermi-level. This solution is in fact just a shift in position of the solution we obtained in Case~1. 
\begin{figure}[ht]\centering
 \subfigure[Quasi potential u]{
  \includegraphics[width=17em]{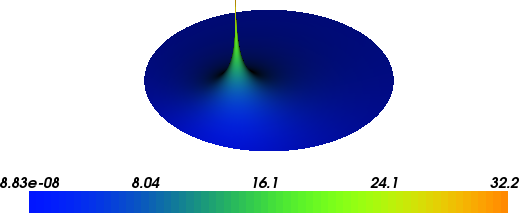}
  \label{fig::nonunique::quasi}
 }
 \subfigure[Potential V]{
  \includegraphics[width=17em]{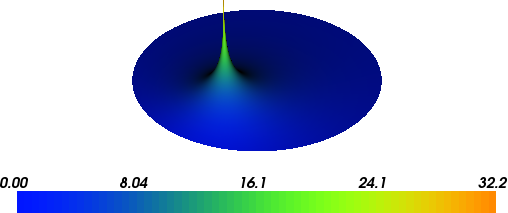}
  \label{fig::nonunique::potential}
 }
 \caption{Another stationary solution $(u,V)$ for $\s=10\pi$ with $F=-20.188$}
\end{figure}

%\section{Concluding Remarks}
%-Insert-

\appendix
\section{Proof of Lemma~\ref{lem::linfty&h1}}
\subsection{Proof of Lemma~\ref{lem::linfty&h1}.1}
 As in Theorem~\ref{thm::existence} we define the function $\phi\colon\rr\to\rr$ by $\phi(s) = (e^{|s|}-1)\,\mbox{sign}(s)$ and let \[G_k(s) = s - T_k(s) = s - \max\{-k,\min\{s,k\}\} = \min\{k-s,\max\{0,s-k\}\}.\] Set $A_k = \{|u|>k\}$. Using $\test_k=\phi(G_k(u))$, $k\geq k_0\geq 1$ as a test function in (\ref{eq::weak}),
 \[
  \frac{\e^2}{4}\int_{A_k}|\nabla G_k(u)|^2|\test_k'|\,\d{x} + \frac{\e^2}{4}\int_{A_k}|\nabla G_k(u)|^2\,\d{x} + \int_{A_k} |u||\test_k|\,\d{x} \leq \int_{A_k} |V||\test_k|\,\d{x}.
 \]
 Note that by definition $\psi(s) = \int_0^{|s|}\sqrt{\phi'(t)}\,\d{t} = 2\,(e^{|s|/2}-1)$. Furthermore, it is easy to see that there exist constants $\mu_1,\mu_2>0$ such that
 \begin{equation}\label{lem::linfty::eq::ineq_psi}
  \mu_1|\psi(s)|^2 \leq |\phi(s)|,\quad |s|\geq 0\quad\mbox{and}\mbox\quad |\phi(s)|\leq \mu_2|\psi(s)|^2,\quad |s|\geq 1.
 \end{equation}
 Writing the left-hand side of the inequality above in terms of $\psi(G_k(u))$ and using the first inequality in (\ref{lem::linfty::eq::ineq_psi}) for half of the third term, we get
 \begin{eqnarray}\label{lem::linfty::eq::ineq1}
  \frac{1}{4}\min\{\e^2,2\mu_1\}\|\psi(G_k(u))\|_{H^1(A_k)}^2 + \frac{\e^2}{4}\|\nabla G_k(u)\|_{L^2(A_k)}^2 && \\
  &&\hspace*{-7em} + \frac{1}{2}\int_{A_k} |u||\test_k|\,\d{x} \leq \int_{A_k} |V||\test_k|\,\d{x} \nonumber.
 \end{eqnarray}
 Now decompose the right-hand side as follows
 \begin{eqnarray*}
  \int_{A_k} |V||\test_k|\,\d{x} &\leq& \int_{(A_k\setminus A_{k+1})\cap\{|V|>1\}} |V||\test_k|\,\d{x} \\
  &&\hspace*{4em} +\int_{A_{k+1}\cap\{|V|>1\}} |V||\test_k|\,\d{x} + \int_{A_k\cap\{|V|\leq 1\}} |\test_k|\,\d{x} \\
  &=& J_1 + J_2 + J_3.
 \end{eqnarray*}
 For $k_0\geq 4$, we can absorb $J_3$ into the left-hand side of (\ref{lem::linfty::eq::ineq1}). As for $J_1$, we have
 \[
  |J_1| \leq |\phi(1)|\int_{A_k\cap \{|V|>1\}}|V|\,\d{x} \leq |\phi(1)|\|V\|_{L^p(\{|V|>1\})}|A_k|^{1/p'}
 \]
 Since $p'\in(1,\frac{2d}{d-2})$ we can use the H\"older, interpolation and Young inequalities, along with the Sobolev embedding $H^1\hookrightarrow L^{\frac{2d}{d-2}}$ to obtain
 \begin{eqnarray*}
  |J_2| &\leq& \|V\|_{L^p(\{|V|>1\})}\|\test_k\|_{L^{p'}(A_{k+1})}\\
  &\leq& \|V\|_{L^p(\{|V|>1\})}\|\test_k\|_{L^1(A_{k+1})}^{1-\frac{d}{2p}}\|\test_k\|_{L^{\frac{d}{d-2}}(A_{k+1})}^{\frac{d}{2p}} \\
  &\leq& c_1\|\psi(G_k(u))\|_{L^\frac{2d}{d-2}(A_k)}^2 + c_2\|V\|_{L^p(\{|V|>1\})}^{\frac{2p}{2p-d}}\|\test_k\|_{L^1(A_k)}\\
  &\leq& c_1c_3\|\psi(G_k(u))\|_{H^1(A_k)}^2 + c_2\|V\|_{L^p(\{|V|>1\})}^{\frac{2p}{2p-d}}\|\test_k\|_{L^1(A_k)}.
 \end{eqnarray*}
 So by choosing $c_1$ sufficiently small and $k_0$ sufficiently large, $J_2$ may also be absorbed into the left-hand side of (\ref{lem::linfty::eq::ineq1}), leaving us with
 \[
  \delta_1\|\psi(G_k(u))\|_{H^1(A_k)}^2 + \delta_2\|G_k(u)\|_{H^1(A_k)}^2 \leq \delta_3|A_k|^{1/p'},
 \]
 for suitable constants $\delta_i>0$, $i=1,2$. Here we used $|s|\leq |\phi(s)|$ for $|s|\geq 0$ and the fact that $|G_k(u)|=|u|-k$ a.e.~on $A_k$.

 Now if $h>k>k_0$, then $A_h\subset A_k$ and for arbitrary $m\in\nn$,
 \[
  \left(\int_{A_k}(|u|-k)^m\,\d{x}\right)^{2/m} \geq \left(\int_{A_k}(h-k)^m\,\d{x}\right)^{2/m} = (h-k)^2|A_h|^{2/m}.
 \]
 Hence, by the Sobolev embedding $H^1\hookrightarrow L^{\frac{2d}{d-2}}$
 \[
  c_3\delta_2(h-k)^2|A_h|^{(d-2)/d} \leq c_3\delta_2\|G_k(u)\|_{L^{\frac{2d}{d-2}}(A_k)}^2 \leq \delta_3|A_k|^{1/p'}.
 \]
 Defining $\zeta(h):= |A_h|^{(d-2)/d}$ and $\beta:= d/p'(d-2)>1$, we finally obtain
 \[
  \zeta(h) \leq \frac{\delta}{(h-k)^2}\,\zeta(k)^\beta,\quad\mbox{for}\;h>k\geq k_0,
 \]
 with $\delta = \delta_3/c_3\delta_2$. From a lemma of Kinderlehrer and Stampacchia (c.f.~\cite[II. Lemma B1]{kinder}), it follows that $\zeta(k)=0$ for every $k\geq K_1\geq k_0$, where
 \[
  K_1 = k_0 + 2^{\beta/(\beta-1)}\zeta(k_0)^{(\beta-1)/2}\sqrt{\delta}.
 \]
 From the definition of $\zeta$ we finally obtain the uniform bound $|u|\leq K_1$ a.e.~on $\Omega$, which gives the bounds required.\qquad\endproof

\subsection{Proof of Lemma~\ref{lem::linfty&h1}.2}
 Following the proof of Lemma~\ref{lem::linfty&h1}.1 we consider $\test_k=\phi(G_k(u))$ for $k\geq k_0(h)=\max\{1,4h\}$, with a suitable $h$ chosen later, as a test function in (\ref{eq::weak}) to obtain
 \begin{eqnarray*}
  \frac{1}{4}\min\{\e^2,2\mu_1\}\|\psi(G_k(u))\|_{H^1(A_k)}^2 + \frac{\e^2}{4}\|\nabla G_k(u)\|_{L^2(A_k)}^2 &&\\
  &&\hspace*{-7em} + \frac{1}{2}\int_{A_k} |u||\test_k|\,\d{x} \leq \int_{A_k} |V||\test_k|\,\d{x}.
 \end{eqnarray*}
 This time we decompose the right-hand side as follows
 \begin{eqnarray*}
  &&\int_{A_k} |V||\test_k|\,\d{x} = \int_{(A_k\setminus A_{k+1})\cap\{|V|>h\}} |V||\test_k|\,\d{x} \\
  &&\hspace*{4em}+ \int_{A_{k+1}\cap\{|V|>h\}} |V||\test_k|\,\d{x} + h\int_{A_k\cap\{|V|\leq h\}} |\test_k|\,\d{x} = J_1 + J_2 + J_3.
 \end{eqnarray*}
 For each of the $J_i$, $i\in\{1,2,3\}$, we have the bounds
 \begin{eqnarray*}
  |J_1| &\leq& |\phi(1)|\int_{\{|V|>h\}}|V|\,\d{x} \leq |\phi(1)|h^{\frac{2-d}{2}}\|V\|_{L^{\frac{d}{2}}(\{|V|>h\})}^{\frac{d}{2}}\\
  |J_2| &\leq& \|V\|_{L^{\frac{d}{2}}(\{|V|>h\})}\|\phi(G_k(u))\|_{L^{\frac{2d}{d-2}}(A_k)} \\
  &\leq& c_1 \mu_2\|V\|_{L^{\frac{d}{2}}(\{|V|>h\})}\|\psi(G_k(u))\|_{H^1(A_k)}^2 \\
  |J_3| &\leq& \frac{1}{4}\int_{A_k} k_0\,|\test_k|\,\d{x} \leq \frac{1}{4}\int_{A_k} |u||\test_k|\,\d{x},
 \end{eqnarray*}
 where we used the Sobolev embedding $H^1\hookrightarrow L^{\frac{2d}{d-2}}$ and the second inequality in (\ref{lem::linfty::eq::ineq_psi}) for $J_2$. By absorbing $J_2$ and $J_3$ into the left-hand side, we obtain
 \begin{equation}\label{lem::h1::eq::ineq1}
  \delta\,\|\psi(G_k(u))\|_{H^1(A_k)}^2 + \frac{\e^2}{4}\|G_k(u)\|_{H^1(A_k)}^2 \leq c_4(h,V),
 \end{equation}
 where, for $h$ large enough,
 \begin{equation}\label{lem::h1::eq::delta}
  \delta = \frac{1}{4}\left(\min\{\e^2,2 \mu_1\}-4 c_1 \mu_2\|V\|_{L^{\frac{d}{2}}(\{|V|>h\})}\right) > 0.
 \end{equation}
 Hence $\psi(G_k(u))\in H^1(\Omega)$ and $\nabla G_k(u)\in [L^2(\Omega)]^2$ for any $k\geq k_0(h)$.

 Now we fix $h$ such that (\ref{lem::h1::eq::delta}) holds and use $\test_{k_0} = \phi(T_{k_0}(u))$ as a test function in the weak formulation (\ref{eq::weak}), to obtain
 \begin{eqnarray*}
  \frac{\e^2}{2}\int_\Omega|\nabla T_{k_0}(u)|^2|\test_{k_0}'|\,\d{x} + \int_\Omega |u||\test_{k_0}|\,\d{x} &\leq& \frac{\e^2}{4}\int_\Omega|\nabla T_{k_0}(u)|^2|\test_{k_0}|\,\d{x} \\
  &&\hspace*{-4em}+ |\phi(k_0(h))|\left[\frac{\e^2}{4}\int_\Omega|\nabla G_{k_0}(u)|^2\,\d{x} + \int_\Omega |V|\,\d{x}\right],
 \end{eqnarray*}
 which leads to 
 \begin{equation}\label{lem::h1::eq::ineq2}
  \frac{\e^2}{4}\int_\Omega|\nabla \psi(T_{k_0}(u))|^2\,\d{x} + \frac{\e^2}{4}\int_\Omega|\nabla T_{k_0}(u)|^2\,\d{x} + \int_\Omega |u||\test_{k_0}|\,\d{x} \leq c_5(h,V).
 \end{equation}
 Combining (\ref{lem::h1::eq::ineq1}) and (\ref{lem::h1::eq::ineq2}), we obtain a uniform bound for the gradient term $\nabla u$,
 \[
  \int_\Omega |\nabla u|^2\,\d{x} = \int_{\{|u|\leq k_0\}} |\nabla T_{k_0}(u)|^2\,\d{x} + \int_{\{|u|>k_0\}} |\nabla G_{k_0}(u)|^2\,\d{x}\leq c_6(h,V)
 \]
 and uniform bound for the gradient term $\nabla e^{|u|/2}$,
 \begin{eqnarray*}
  \int_\Omega |\nabla e^{|u|/2}|^2\,\d{x} &=& \frac{1}{4}\left[\int_{\{|u|\leq k_0\}} |\nabla \psi(T_{k_0}(u))|^2\,\d{x} + e^{k_0}\int_{\{|u|>k_0\}} |\nabla \psi(G_{k_0}(u))|^2\,\d{x}\right] \\
  &\leq& c_7(h,V).
 \end{eqnarray*}
 Hence, Lemma~\ref{lem::linfty&h1}.1 and the estimates above give us constants $M_1,M_2>0$ such that
 \[
  \|u\|_{H^1(\Omega)}\leq M_1\quad\mbox{and}\quad \|\nabla e^{|u|/2}\|_{H^1(\Omega)}\leq M_2,
 \]
 which concludes the proof.\qquad\endproof

\bibliographystyle{siam}
\bibliography{biblio} 

\begin{thebibliography}{10}

\bibitem{abdallah}
{\sc N.~Ben Abdallah and A.~Unterreiter}, {\em {On the stationary quantum drift
  diffusion model}}, {Z. angew. Math. Phys.}, 49 (1998), pp.~251--275.

\bibitem{bennett}
{\sc C.~Bennett and R.~C. Sharpley}, {\em {Interpolation of Operators}},
  no.~v.~129 in {Pure and Applied Mathematics}, {Academic Press}, 1988.

\bibitem{biler}
{\sc P.~Biler and T.~Nadzieja}, {\em {Existence and nonexistence of solutions
  for a model of gravitational interaction of particles, I}}, {Colloquium
  Mathematicum}, LXVI (1993).

\bibitem{caglioti1}
{\sc E.~Caglioti, P.L. Lions, C.~Marchioro, and M.~Pulvirenti}, {\em {A special
  class of stationary flows for two-dimensional Euler equations: A statistical
  mechanics description. Part I}}, {Commun. Math. Phys.}, 143 (1992),
  pp.~501--525.

\bibitem{caglioti2}
\leavevmode\vrule height 2pt depth -1.6pt width 23pt, {\em {A special class of
  stationary flows for two-dimensional Euler equations: A statistical mechanics
  description. Part II}}, {Commun. Math. Phys.}, 174 (1995), pp.~229--260.

\bibitem{cassani}
{\sc D.~Cassani, B.~Ruf, and C.~Tarsi}, {\em {Best constants in a boderline
  case of second-order Moser type inequalities}}, {Ann. I. H. Poincar\'e}, 27
  (2010), pp.~73--93.

\bibitem{dall}
{\sc A.~Dall'Aglio, D.~Giachetti, and J.~P. Puel}, {\em {Nonlinear elliptic
  equations with natural growth in general domains}}, {Annali di Matematica},
  181 (2002), pp.~407--426.

\bibitem{ferone}
{\sc V.~Ferone, M.~R. Posteraro, and J.~M. Rakotoson}, {\em
  {$L^\infty$-estimates for nonlinear elliptic problems with $p$-growth in the
  gradient}}, {J. Inequal. Appl.}, 3 (1999), pp.~109--125.

\bibitem{gilbarg}
{\sc D.~Gilbarg and N.~S. Trudinger}, {\em {Elliptic partial differential
  equations of second order}}, {Springer--Verlag, Berlin}, 1~ed., 1983.

\bibitem{kinder}
{\sc D.~Kinderlehrer and G.~Stampacchia}, {\em An introduction to variational
  inequalities and their applications}, Pure and Applied Mathematics 88,
  Academic Press, New York, 1980.

\bibitem{leray}
{\sc J.~Leray and J.~L. Lions}, {\em {Quelques r\'esultats de Vi\v{s}ik sur les
  probl\`emes elliptiques non lin\'eaires par les m\'ethodes de
  Minty--Browder}}, {Bull. Soc. Math. France}, 93 (1965), pp.~97--107.

\bibitem{montenegro}
{\sc M.~Montenegro and M.~Montenegro}, {\em {Existence and nonexistence of
  solutions for quasilinear elliptic equations}}, {J. Math. Anal. Appl.}, 245
  (2000), pp.~303--316.

\bibitem{moser}
{\sc J.~Moser}, {\em {A sharp form of an inequality by N.~Trudinger}}, {J.
  Indiana. Univ. Math.}, 20 (1971), pp.~1077--1092.

\bibitem{pinnau}
{\sc R.~Pinnau and A.~Unterreiter}, {\em {The stationary current-voltage
  characteristics of the quantum drift-diffusion model}}, {SIAM J. Num. Anal.},
  37 (2000), pp.~211--245.

\bibitem{suzuki}
{\sc T.~Suzuki}, {\em {Free Energy and Self-Interacting Particles}},
  {Birkh\"auser, Boston}, 2005.

\end{thebibliography}

\end{document}